\documentclass[a4paper,10pt]{amsart}

\usepackage[utf8]{inputenc} 
\usepackage{amsmath}
\usepackage{amsbsy}
\usepackage{amssymb}
\usepackage{amscd,amsthm}
\usepackage{verbatim}
\usepackage[english]{babel}
\usepackage{dsfont}
\usepackage{anysize}
\usepackage{hyperref}

\newcommand{\N}{\mathds{N}}

\newcommand{\p}{\phantom}
\newcommand{\q}{\quad}

\newtheorem{thm}{Theorem}[section]
\newtheorem{lem}[thm]{Lemma}
\newtheorem{kor}[thm]{Corollary}

\newtheorem{prop}[thm]{Proposition}

\theoremstyle{definition}

\theoremstyle{remark}
\newtheorem*{rema}{Remark}

\title{New estimates for some functions defined over primes}
\author{Christian Axler}
\address{Institute of Mathematics\\ Heinrich-Heine University Duesseldorf\\
40225 Duesseldorf, Germany}
\email{christian.axler@hhu.de}
\date{\today}

\begin{document}

\begin{abstract}
In this paper we first establish new explicit estimates for Chebyshev's $\vartheta$-function. Applying these new estimates, we derive new upper and lower 
bounds for some functions defined over the prime numbers, for instance the prime counting function $\pi(x)$, which improve the currently best ones. Furthermore, 
we use the obtained estimates for the prime counting function to give two new results concerning the existence of prime numbers in short intervals.
\end{abstract}

\maketitle

\section{Introduction}

Let $\pi(x)$ denotes the number of primes not exceeding $x$. Since there are infinitely many primes, we have $\pi(x) \to \infty$ for $x \to \infty$. In 1793, 
Gauß \cite{gauss} stated a conjecture concerning the asymptotic behaviour for the prime counting function $\pi(x)$, namely
\begin{equation}
\pi(x) \sim \text{li}(x) \q\q (x \to \infty), \tag{1.1} \label{1.1}
\end{equation}
where the \emph{logarithmic integral} $\text{li}(x)$ defined for every real $x \geq 0$ as
\begin{equation}
\text{li}(x) = \int_0^x \frac{dt}{\log t} = \lim_{\varepsilon \to 0} \left \{ \int_{0}^{1-\varepsilon}{\frac{dt}{\log t}} + 
\int_{1+\varepsilon}^{x}{\frac{dt}{\log t}} \right \}. \tag{1.2} \label{1.2}
\end{equation}
The asymptotic formula \eqref{1.1} was proved by Hadamard \cite{hadamard1896} and, independently, by de la Vall\'{e}e-Poussin \cite{vallee1896} in 1896, and is
known as the \textit{Prime Number Theorem}. In his later paper \cite{vallee1899}, where he proved the existence of a zero-free region for the Riemann 
zeta-function $\zeta(s)$ to the left of the line $\text{Re}(s) = 1$, de la Vall\'{e}e-Poussin also estimated the error term in the Prime Number Theorem by 
showing
\begin{equation}
\pi(x) = \text{li}(x) + O(x \exp(-c_0\sqrt{\log x})), \tag{1.3} \label{1.3}
\end{equation} 
where $c_0$ is a positive absolute constant. The work of  Korobov \cite{korobov1958} and Vinogradov \cite{vinogradov1958} implies the currently best error 
term, namely that there is a positive absolute constant $c_1$ so that
\begin{equation}
\pi(x) = \text{li}(x) + O \left( x \exp \left( - c_1 (\log x)^{3/5} (\log \log x)^{-1/5} \right) \right). \tag{1.4} \label{1.4}
\end{equation}
The computation of the prime counting function $\pi(x)$ for large values of $x$ is a difficult problem (the latest record was $\pi(10^{25}) =  
176\,846\,309\,399\,143\,769\,411\,680$ and is due to Büthe, Franke and Kleinjung \cite{buethepre2}). Since the asymptotic formula \eqref{1.4} is not very 
meaningful with regard to the computation of $\pi(x)$ for some fixed $x$, we are interested to find new explicit estimates for the prime counting function. In 
order to do this, we first need to establish the following result on Chebyshev's $\vartheta$-function
\begin{displaymath}
\vartheta(x) = \sum_{p \leq x} \log p,
\end{displaymath}
which improves several known estimates for this function.

\begin{thm} [See Theorem \ref{thm204}] \label{thm101}
For every $x \geq 19\,035\,709\,163$, we have
\begin{displaymath}
\vartheta(x) > x - \frac{0.15x}{\log^3 x},
\end{displaymath}
and for every $x > 1$, we have
\begin{displaymath}
\vartheta(x) < x + \frac{0.15x}{\log^3 x}.
\end{displaymath}
\end{thm}

In 2000, Panaitopol \cite[p. 55]{pan3} gave another asymptotic formula for the prime counting function by showing that for each positive integer $m$, we have
\begin{displaymath}
\pi(x) = \frac{x}{ \log x - 1 - \frac{k_1}{\log x} - \frac{k_2}{\log^2x} - \ldots -  \frac{k_{m}}{\log^{m} x}} + O \left( \frac{x}{\log^{m+2}x}  \right),
\end{displaymath}
where the positive integers $k_1, \ldots, k_m$ are defined by the recurrence formula
\begin{displaymath}
k_m + 1!k_{m-1} + 2!k_{m-2} + \ldots + (m-1)!k_1 = m \cdot m!.
\end{displaymath}
In Section 3, we use the inequalities obtained in Theorem \ref{thm101} to find among others the following explicit estimates for the prime counting function,
which improve the current best estimates for $\pi(x)$.

\begin{thm} [See Theorem \ref{thm302}] \label{thm102}
For every $x \geq 49$, we have
\begin{displaymath}
\pi(x) < \frac{x}{\log x - 1 - \frac{1}{\log x} - \frac{3.15}{\log^2x} - \frac{12.85}{\log^3x} - \frac{71.3}{\log^4x} - \frac{463.2275}{\log^5x} -
\frac{4585}{\log^6x}}.
\end{displaymath}
\end{thm}

\begin{thm} [See Theorem \ref{thm308}] \label{thm103}
For every $x \geq 19\,033\,744\,403$, we have
\begin{displaymath}
\pi(x) > \frac{x}{\log x - 1 - \frac{1}{\log x} - \frac{2.85}{\log^2x} - \frac{13.15}{\log^3x} - \frac{70.7}{\log^4x} - \frac{458.7275}{\log^5x} -
\frac{3428.7225}{\log^6x}}.
\end{displaymath}
\end{thm}

In Section 4, we apply these new estimates for the prime counting function to derive two new result concerning the existence of prime numbers in short 
intervals. The origin of this problem is Bertrand's postulate, which states that for each positive integer $n$ there is a prime number $p$ with $n < p \leq 
2n$. We give the following both refinements.

\begin{thm} [See Theorem \ref{thm401}] \label{thm104}
For every $x \geq 6\,034\,256$ there is a prime number $p$, such that
\begin{displaymath}
x < p \leq x \left( 1 + \frac{0.087}{\log^3x} \right),
\end{displaymath}
and for every $x > 1$ there is a prime number $p$, such that
\begin{displaymath}
x < p \leq x \left( 1 + \frac{198.2}{\log^4x} \right).
\end{displaymath}
\end{thm}

In Section 5 and Section 6, we use Theorem \ref{thm101} to derived some upper and lower bounds for the prime functions
\begin{displaymath}
\sum_{p \leq x} \frac{1}{p}, \q \sum_{p \leq x} \frac{\log p}{p}, \q \prod_{p \leq x} \left( 1 - \frac{1}{p} \right),
\end{displaymath}
which improves Dusart's \cite{dusart2017} estimates for these functions.

\section{New estimates for Chebyshev's $\vartheta$-function}

The prime counting function $\pi(x)$ and Chebyshev's $\vartheta$-function are connected by the identities
\begin{equation}
\pi(x) = \frac{\vartheta(x)}{\log x} + \int_{2}^{x}{\frac{\vartheta(t)}{t \log^{2} t}\ dt} \tag{2.1} \label{2.1}
\end{equation}
and
\begin{equation}
\vartheta(x) = \pi(x) \log x - \int_{2}^{x}{\frac{\pi(t)}{t}\ dt}, \tag{2.2} \label{2.2}
\end{equation}
which hold for every $x \geq 2$ (see, for instance, Apostol \cite[Theorem 4.3]{ap}). In order to find new estimates for the prime counting function, we first
derive some new upper and lower bounds for Chebyshev's $\vartheta$-function and then use \eqref{2.1}. Using \eqref{2.2}, it is easy to see that the Prime 
Number Theorem \eqref{1.1} is equivalent to
\begin{equation}
\vartheta(x) \sim x \q\q (x \to \infty). \tag{2.3} \label{2.3}
\end{equation}
By proving the existence of a zero-free region for the Riemann zeta-function, de la Vall\'{e}e-Poussin \cite{vallee1899} was able to bound the error term in 
\eqref{2.3} by
\begin{equation}
\vartheta(x) = x + O(x \exp(-c_2\sqrt{\log x})), \tag{2.4} \label{2.4}
\end{equation} 
where $c_2$ is a positive absolute constant. 
The asymptotic formula \eqref{2.4} implies that for every positive integer $k$ there is a positive real number $\eta_k$ and a real $x_0(k) \geq 2$ so that
\begin{equation}
|\vartheta(x) - x| < \frac{\eta_kx}{\log^k x} \tag{2.5} \label{2.5}
\end{equation}
for every $x \geq x_0(k)$.
The work of Korobov \cite{korobov1958} and Vinogradov \cite{vinogradov1958} imply the best known error term in \eqref{2.3} namely
\begin{displaymath}
\vartheta(x) = x + O \left( x \exp \left( - c_3 \log^{3/5}x (\log \log x)^{-1/5} \right) \right),
\end{displaymath}
where $c_3$ is an absolute positive constant. Under the assumption that the Riemann hypothesis is true, von Koch \cite{koch1901} deduced the improved 
asymptotic formula
\begin{displaymath}
\vartheta(x) = x + O(\sqrt{x} \log^2 x).
\end{displaymath}
A precise version of this was given by Schoenfeld \cite[Theorem 10]{schoenfeld1976}. He found under the assumption that the Riemann hypothesis is true that
\begin{equation}
|\vartheta(x) - x| < \frac{1}{8\pi} \, \sqrt{x} \log^2 x \tag{2.6} \label{2.6}
\end{equation}
for every $x \geq 599$. In 2014, Büthe \cite[p. 2495]{buethe2016} found that the inequality \eqref{2.6} holds unconditionally for every $x$ such that $599 \leq
x \leq 1.4 \cdot 10^{25}$ by using the following lemma.

\begin{lem}[Büthe, \cite{buethe2016}] \label{lem201}
Let $T > 0$ be such that the Riemann hypthesis holds for every $0 < \emph{Im}(\rho) \leq T$. Then, under the condition $4.92 \sqrt{x/\log x} \leq T$, the
following estimates hold:
\begin{enumerate}
\item[(a)] $\displaystyle |\vartheta(x) - x| < \frac{1}{8\pi} \, \sqrt{x} \log^2 x$ for every $x \geq 599$,
\item[(b)] $\displaystyle |\pi(x) - \emph{li}(x)| < \frac{1}{8\pi} \, \sqrt{x} \log x$ for every $x \geq 2\,657$.
\end{enumerate}
\end{lem}

In the following proposition we also make use of Lemma \ref{lem201} to increase the number $1.4 \cdot 10^{25}$ in Büthe's result on \eqref{2.6}.

\begin{prop} \label{prop202}
The inequality \eqref{2.6} holds unconditionally for every $x$ such that $599 \leq x \leq 5.5 \cdot 10^{25}$.
\end{prop}

\begin{proof}
Let $N(T)$ be the number of complex zeros $\rho$ of the Riemann zeta function $\zeta(s)$ satisfying $0 < \text{Im}(\rho) \leq T$. Trudgian \cite[Corollary
1]{trudgian2014} found that $N(T)$ is bounded by
\begin{equation}
N(T) \leq \frac{T}{2\pi} \log \frac{T}{2\pi e} + \frac{7}{8} + 0.112 \log T + 0.278 \log \log T + 2.51 + \frac{0.2}{T} \tag{2.7} \label{2.7}
\end{equation}
for every $T \geq e$. Setting $T_0 = 4\,768\,099\,715\,087$, we use \eqref{2.7} to get
\begin{equation}
N(T_0) \leq 2 \cdot 10^{13}; \tag{2.8} \label{2.8}
\end{equation}
i.e. there are at most $2 \cdot 10^{13}$ complex zeros $\rho$ of the Riemann zeta function $\zeta(s)$ satisfying $0 < \text{Im}(\rho) \leq T_0$. By 
\cite{gourdon}, 
the first $2 \cdot 10^{13}$ zeros of the Riemann zeta function satisfy the Riemann hypothesis. Together with 
\eqref{2.8}, we 
obtain that the Riemann hypothesis holds for every complex zeros $\rho$ such that $0 < \text{Im}(\rho) \leq T_0$. Now, we set $x_0 = 5.5 \cdot 10^{25}$ to get 
$4.92 \sqrt{x_0/\log x_0} \leq T_0$ and it remains to apply Lemma \ref{lem201}.
\end{proof}

In the direction of \eqref{2.5}, Dusart found in \cite[Theorem 5.2]{dusart2010} and \cite[Theorem 4.2]{dusart2017} the following explicit estimates for the
distance between $x$ and $\vartheta(x)$.

\begin{lem}[Dusart, \cite{dusart2010}, \cite{dusart2017}] \label{lem203}
We have
\begin{displaymath}
\vert \vartheta(x) - x \vert < \frac{\eta_k x}{\log^{k} x}
\end{displaymath}
for every $x \geq x_0(k)$ with
\begin{center}
\begin{tabular}{|l||c|c|c|c|c|c|c|c|c|c|}
\hline
$k$\rule{0mm}{4mm}         & $             1$ & $           2$ &$2$ & $3$& $        3$ & $4$  \\ \hline
$\eta_k$\rule{0mm}{4mm} & $       0.001$ & $        0.2$ &$0.01$ & $1$&$   0.5  $ & $151.3$  \\ \hline
$x_0(k)$\rule{0mm}{4mm}& $908\,994\,923$ & $3\,594\,641$ & $7\,713\,133\,853$ & $89\,967\,803$ & $767\,135\,587$ & $2$ \\ \hline
\end{tabular} \ .
\end{center}
\vspace{1mm}
\end{lem}

In the following theorem, we find the corresponding value $x_0$ for the case $k = 3$ and $\eta_3 = 0.15$. In the proof, we use explicit estimates for 
Chebyshev's $\psi$-function, which is defined by 
\begin{displaymath}
\psi(x) = \sum_{p^m \leq x} \log p.
\end{displaymath}

\begin{thm} \label{thm204}
For every $x \geq 19\,035\,709\,163 = p_{841\,508\,302}$, we have
\begin{equation}
\vartheta(x) > x - \frac{0.15x}{\log^3 x}, \tag{2.9} \label{2.9}
\end{equation}
and for every $x > 1$, we have
\begin{equation}
\vartheta(x) < x + \frac{0.15x}{\log^3 x}. \tag{2.10} \label{2.10}
\end{equation}
\end{thm}

\begin{proof}
First, we check that the inequality
\begin{equation}
|\vartheta(x) - x| < \frac{0.15x}{\log^3 x} \tag{2.11} \label{2.11}
\end{equation}
holds for every $x \geq e^{32}$. By Dusart \cite[Corollary 1.2]{dusart2016}, we have
\begin{equation}
|\vartheta(x) - x| < \frac{\sqrt{8}}{\sqrt{\pi \sqrt{R}}} \, x(\log x)^{1/4} e^{-\sqrt{(\log x)/R}} \tag{2.12} \label{2.12}
\end{equation}
for every $x \geq 3$, where $R = 5.69693$. Since $g(x) = (\log x)^{13/4} e^{-\sqrt{(\log x)/R}}$ is a monotonic decreasing function for every $x 
\geq e^{169R/4}$, we get that
\begin{displaymath}
|\vartheta(x) - x| < \frac{\sqrt{8}}{\sqrt{\pi \sqrt{R}}} \; g(e^{5\,000})\; \frac{x}{\log^3 x} < \frac{0.148x}{\log^3 x}
\end{displaymath}
for every $x \geq e^{5000}$. Using \cite[Corollary 4.5]{dusart2017}, we get
\begin{equation}
|\vartheta(x) - x| < \left( \frac{(1 + 1.47 \cdot 10^{-7}) b_i^3}{\sqrt{e^{b_i}}} + \frac{1.78b_i^3}{\sqrt[3]{e^{2b_i}}} +\varepsilon_i b_{i+1}^3 \right) 
\frac{x}{\log^3 x} \tag{2.13} \label{2.13}
\end{equation}
for $e^{b_i} \leq x \leq e^{b_{i+1}}$, where $b_i$ and the corresponding $\varepsilon_i$ are given in Table 5.2 of \cite{dusart2017}. Substituting $b_{31} = 
1\,500, b_{32} = 2\,000, b_{33} = 2\,500, b_{34} = 3\,000, b_{35} = 3\,500, b_{36} = 4\,000, b_{37} = 4\,500$ and the corresponding values of $\varepsilon_i$ in 
\eqref{2.13}, we obtain that the inequality \eqref{2.11} also holds for every $e^{1500} \leq x \leq e^{5000}$. From Tables 6.4 and 6.5 of \cite{dusart2010}, it 
follows that the inequality \eqref{2.11} holds for every $x$ such that $e^{32} \leq x < e^{1500}$. 

So, to prove that \eqref{2.9} holds for every $x \geq 19\,035\,709\,163$, it remains to deal with the case where $19\,035\,709\,163 \leq x \leq e^{32}$. By 
Büthe \cite[Theorem 2]{buethe}, we have $\vartheta(t) \geq t - 1.95\sqrt{t}$ for every $t$ such that $1\,423 \leq t \leq 10^{19}$. Since $0.15\sqrt{t} > 1.95 
\log^3 t$ for every $t \geq 34\,485\,879\,392$, we get that the inequality \eqref{2.9} holds also for every $x$ such that $34\,485\,879\,392 \leq x \leq 
e^{32}$. In addition, Büthe \cite[p. 13]{buethe} found that $-0.8 \leq (t - \psi(t))/\sqrt{t} \leq 0.81$ for every $t$ such that $100 \leq t \leq 5 \cdot 
10^{10}$. Now, we use Lemma 1 of \cite{buethe} to get
\begin{displaymath}
\vartheta(t) \geq t - 1.81\sqrt{t} - 0.8t^{1/4} - 1.03883(t^{1/3} + t^{1/5} + 2 t^{1/13}\log t)
\end{displaymath}
for every $t$ such that $10\,000 \leq t \leq 5 \cdot 10^{10}$. Since $t^{1/5} + 2 t^{1/13}\log t \leq t^{1/3}$ for every $t \geq 783\,674$, we get
\begin{equation}
\vartheta(t) \geq t - 1.81\sqrt{t} - 0.8t^{1/4} - 2 \cdot 1.03883t^{1/3} \tag{2.14} \label{2.14}
\end{equation}
for every $t$ such that $783\,674 \leq t \leq 5 \cdot 10^{10}$. Now, we notice that $0.15t/\log^3 t \geq 1.81\sqrt{t} + 0.8t^{1/4} + 2 \cdot 1.03883t^{1/3}$ 
for every $t \geq 29\,946\,085\,320$. Hence, by \eqref{2.14}, the inequality \eqref{2.9} is fulfilled for every $x$ such that $29\,946\,085\,320 \leq x \leq 
34\,485\,879\,392$ as well. To prove that the inequality \eqref{2.9} is also valid for every $x$ such that $19\,035\,709\,163 \leq x < 29\,946\,085\,320$, we 
set $f(x) = x(1-0.15/\log^3x)$. Since $f$ is a strictly increasing function on $(1, \infty)$, it suffices to check with a computer that $\vartheta(p_n) > 
f(p_{n+1})$ for every positive integer $n$ such that $\pi(19\,035\,709\,163) \leq n \leq \pi(29\,946\,085\,320)$.

Now, we show that \eqref{2.10} for every $x > 1$. Using the inequality \eqref{2.11}, it suffices to prove that \eqref{2.10} holds for every $x$ such that $1 < 
x < e^{32}$. For this, we use another result of Büthe \cite[Theorem 2]{buethe}. He found that $\vartheta(x) < x$ for every $x$ such that $1 \leq x \leq 
10^{19}$, which clearly implies that the inequality \eqref{2.10} holds for every $x$ such that $1 < x < e^{32}$.
\end{proof}

\begin{rema}
In \cite[Proposition 3.2]{axler2016} it is shown that $\vartheta(x) > x - 0.35x/\log^3 x$ for every $x \geq e^{30}$. Using Theorem \ref{thm204}, we get that 
this inequality also holds for every $x$ such that $19\,035\,709\,163 \leq x \leq e^{30}$. A computer check gives that the inequality
\begin{displaymath}
\vartheta(x) > x - \frac{0.35x}{\log^3 x}
\end{displaymath}
holds for every $x \geq 1\,332\,492\,593$.
\end{rema}

In the next proposition, we give a slightly improvement of Lemma \ref{lem203} for the case $k = 4$, which improves the inequality \eqref{2.9} for every $x \geq
e^{666 + 2/3}$.

\begin{prop} \label{prop205}
For every $x \geq 70\,111$, we have
\begin{equation}
|\vartheta(x) - x| < \frac{100x}{\log^4 x}. \tag{2.15} \label{2.15}
\end{equation}
\end{prop}

\begin{proof}
Let $R = 5.69693$. We use \eqref{2.12} to get that $|\vartheta(x) - x| < 100x/\log^4 x$ for every $x \geq e^{6000}$.
Similarly to the proof of Theorem \ref{thm204}, we check with Table 5.2 of \cite{dusart2017} that the inequality \eqref{2.15} holds for every $x$ such that 
$e^{1000} \leq x \leq e^{6000}$ as well. From Tables 6.4 and 6.5 of \cite{dusart2010}, it follows that the required inequality holds for every $x$ such that 
$e^{23} \leq x < e^{1000}$. Finally, we obtain that $1/\log^3t < 100/\log^4t$ for every $t$ such that $1 < t \leq e^{100}$ and so, Lemma \ref{lem203} implies 
the validity of the required inequality for every $x$ such that $89\,967\,803 \leq x \leq e^{23}$. To prove that the inequality \eqref{2.15} is also fulfilled 
for every $x$ such that $70\,111 \leq x < 89\,967\,803$, we set $f(x) = x(1-100/\log^4x)$. Since $f$ is a strictly increasing function for every $x > 1$, it is 
enough to check with a computer that $\vartheta(p_n) > f(p_{n+1})$ for every positive integer $n$ such that $\pi(70\,111) \leq n \leq \pi(89\,967\,803)$.
\end{proof}

\section{New estimates for the prime counting function}

Under the assumption that the Riemann hypothesis is true, von Koch \cite{koch1901} deduced a remarkable refinement of error term in the Prime Theorem, which is
given by
\begin{displaymath}
\pi(x) = \text{li}(x) + O(\sqrt{x} \log x).
\end{displaymath}
A precise version of Koch's result is due to Schoenfeld \cite[Corollary 1]{schoenfeld1976}. He found under the assumption that the Riemann hypothesis is true 
that the inequality
\begin{equation}
|\pi(x) - \text{li}(x)| < \frac{1}{8\pi}\, \sqrt{x} \log x \tag{3.1} \label{3.1}
\end{equation}
holds for every $x \geq 2\,657$. In 2014, Büthe \cite[p. 2495]{buethe2016} showed that the inequality \eqref{3.1} holds unconditionally for every $x$ such that
$2\,657 \leq x \leq 1.4 \cdot 10^{25}$. The following proposition gives a slightly improvement of Büthe's result.

\begin{prop} \label{prop301}
The inequality \eqref{3.1} holds unconditionally for every $x$ such that $2\,657 \leq x \leq 5.5 \cdot 10^{25}$.
\end{prop}

\begin{proof}
Similar to the proof of Proposition \ref{prop202}.
\end{proof}

Now, let $k$ be a positive integer and let $\eta_k, x_1(k)$ be positive real numbers with $x_1(k) \geq 2$ so that
\begin{displaymath}
\vert \vartheta(x) - x \vert < \frac{\eta_k x}{\log^{k} x}
\end{displaymath}
for every $x \geq x_1(k)$. Together with \eqref{2.1}, we get
\begin{equation}
J_{k,-\eta_k,x_1(k)}(x) \leq \pi(x) \leq J_{k,\eta_k,x_1(k)}(x) \tag{3.2} \label{3.2}
\end{equation}
for every $x \geq x_1(k)$, where
\begin{displaymath}
J_{k,\eta_k,x_1(k)}(x) = \pi(x_1(k)) - \frac{\vartheta(x_1(k))}{\log x_1(k)} + \frac{x}{\log x} + \frac{\eta_k x}{\log^{k+1} x} + \int_{x_1(k)}^{x}{\left(
\frac{1}{\log^{2} t} + \frac{\eta_k}{\log^{k+2} t} \ dt \right)}.
\end{displaymath}
The function $J_{k,\eta_k,x_1(k)}$ was already introduced by Rosser and Schoenfeld \cite[p.81]{rosser1962} (for the case $k=1$) and Dusart \cite[p. 
9]{dusart2010}. In this section, we use \eqref{3.2} and the estimates for Chebyshev's $\vartheta$-function obtained in the previous section to establish new 
explicit estimates for the prime counting function $\pi(x)$.

\subsection{New upper bounds for the prime counting function}

First we recall that Panaitopol \cite[p. 55]{pan3} gave the asymptotic formula
\begin{displaymath}
\pi(x) = \frac{x}{ \log x - 1 - \frac{k_1}{\log x} - \frac{k_2}{\log^2x} - \ldots -  \frac{k_{m}}{\log^{m} x}} + O \left( \frac{x}{\log^{m+2}x}  \right),
\end{displaymath}
where $m$ is a positive integer and $k_1, \ldots, k_m$ are defined by the recurrence formula
\begin{displaymath}
k_m + 1!k_{m-1} + 2!k_{m-2} + \ldots + (m-1)!k_1 = m \cdot m!.
\end{displaymath}
For instance, we have
\begin{displaymath}
\pi(x) = \frac{x}{ \log x - 1 - \frac{1}{\log x} - \frac{3}{\log^2x} - \frac{13}{\log^3x} - \frac{71}{\log^4x} - \frac{461}{\log^5x} - \frac{3441}{\log^6x}} +
O \left( \frac{x}{\log^8x}  \right).
\end{displaymath}
In this direction, Theorem \ref{thm204} implies the following upper bound for the prime counting function.

\begin{thm} \label{thm302}
For every $x \geq 49$, we have
\begin{equation}
\pi(x) < \frac{x}{\log x - 1 - \frac{1}{\log x} - \frac{3.15}{\log^2x} - \frac{12.85}{\log^3x} - \frac{71.3}{\log^4x} - \frac{463.2275}{\log^5x} - 
\frac{4585}{\log^6x}}. \tag{3.3} \label{3.3}
\end{equation}
\end{thm}

\begin{proof}
Let $x_1 = 10^{15}$,  let $f(x)$ be given by the right-hand side of \eqref{3.3}, and let $r(x)$ be the denominator of $f(x)$. By \eqref{3.2} and Theorem 
\ref{thm204}, we get $\pi(x) \leq J_{3, 0.15, x_1}(x)$ for every $x \geq x_1$. In the first step of the proof, we compare $f(x)$ with $J_{3, 0.15, x_1}(x)$. In 
order to prove that the function $g(x) = f(x) - J_{3, 0.15, x_1}(x)$ is positive for every $x \geq x_1$, we need to show that $g(x_1) > 0$ and that the 
derivative of $g$ is positive for every $x \geq x_1$. By Dusart \cite[Table 6.2]{dusart2010}, we have $\vartheta(x_1) \geq 999\,999\,965\,752\,660$. Further, 
$\pi(x_1) = 29\,844\,570\,422\,669$ and so we compute that $g(x_1) \geq 3 \cdot 10^6$.
To show that the derivative of $g$ is positive for every $x \geq x_1$, we set
\begin{align*}
h_1(y) & = 1119.6775y^{11} - 38212.4575y^{10} - 13858.278375y^9 - 45007.842875y^8 - 189106.352125y^7 \\
& \p{\q\q} - 865668.98286875y^6 - 4248412.96105y^5 - 21029165.2496875y^4 - 47509.2738384375y^3 \\
& \p{\q\q} - 246389.1037096875y^2 - 1241825.47125y
\end{align*}
and compute that $h_1(y) > 0$ for every $y \geq 34.525$. Therefore, we get that the inequality $g'(x) = (h_1(\log x) + 9460001.25)/(r^2(x)\log^{17}x) > 0$ 
holds for every $x \geq x_1$. So, $f(x) - J_{3, 0.15, x_1}(x) = g(x) > 0$ for every $x \geq x_1$ and, by \eqref{3.2}, we conclude that the inequality 
\eqref{3.3} holds for every $x \geq x_1$.

In the second step, we check \eqref{3.3} for every $x$ such that $1\,095\,698 \leq x \leq 10^{15}$ by comparing $f(x)$ with the logarithmic integral 
$\text{li}(x)$. For this, we set
\begin{align*}
h_2(y) & = 0.15y^{11} - 0.75y^{10} + 0.75y^9 - 0.195y^8 + 1118.8525y^7 - 38220.7675y^6 - 13920.74325y^5\\
& \p{\q\q} - 45874.13675y^4 - 183890.7415y^3 - 868400.71675625y^2 - 4247796.175y - 21022225.
\end{align*}
Then, it is easy to see that $h_2(y) \geq 0$ for every $y \geq 12.2714$. Hence, for every $x \geq 213\,502$, we have $f'(x_1) -  \text{li}'(x) = 
h_2(\log x)/(r^2(x) \log^{13}x) \geq 0$.
In addition, we have $f(1\,095\,698) - \text{li}(1\,095\,698) > 0$. Hence $f(x) > \text{li}(x)$ for every $x \geq 1\,095\,698$. Now we use a result of Büthe 
\cite[Theorem 2]{buethe}, namely that
\begin{equation}
\pi(x) < \text{li}(x) \tag{3.4} \label{3.4}
\end{equation}
for every $x$ such that $2 \leq x \leq 10^{19}$, to obtain that the required inequality holds for $x$ such that $1\,095\,698 \leq x \leq 10^{15}$ as well. To 
deal with the case where $101 \leq x \leq 1\,095\,698$, we notice that $f(x)$ is strictly increasing for every $x$ such that $101 \leq x \leq 1\,095\,698$. So 
we check with a computer that $f(p_n) > \pi(p_n)$ for every positive integer $n$ such that $\pi(101) \leq n \leq \pi(1\,095\,698) + 1$. A computer check for 
smaller values of $x$ completes the proof.
\end{proof}

We obtain the following weaker but more compact upper bounds.

\begin{kor} \label{kor303}
We have
\begin{displaymath}
\pi(x) < \frac{x}{\log x - 1 - \frac{a_1}{\log x} - \frac{a_2}{\log^2x} - \frac{a_3}{\log^3x} - \frac{a_4}{\log^4x} - \frac{a_5}{\log^5x}}
\end{displaymath}
for every $x \geq x_0$, where
\begin{center}
\begin{tabular}{|l||c|c|c|c|c|c|c|c|c|c|}
\hline
$a_2$\rule{0mm}{4mm} & $   1     $ & $  1    $ & $  1    $ & $1$ & $1.15$\\ \hline
$a_2$\rule{0mm}{4mm} & $   3.15$ & $  3.15$ & $  3.15$ & $3.69$ & $0$\\ \hline
$a_3$\rule{0mm}{4mm} & $  12.85$ & $12.85$ & $14.21$ & $0$ & $0$\\ \hline 
$a_4$\rule{0mm}{4mm} & $  71.3 $ & $80.43$ & $   0    $ & $0$ & $0$\\ \hline
$a_5$\rule{0mm}{4mm} & $540.59$ & $ 0     $ & $   0    $ & $0$ & $0$\\ \hline
$x_0$\rule{0mm}{4mm} & $ 32     $ & $22    $ & $ 14     $ & $10\,031\,975\,087$  & $38\,284\,442\,297$\\ \hline
\end{tabular} \ .
\end{center}
\end{kor}

\begin{proof}
We only show that the inequality
\begin{equation}
\pi(x) < \frac{x}{\log x - 1 - \frac{1}{\log x} - \frac{3.15}{\log^2x} - \frac{12.85}{\log^3x} - \frac{71.3}{\log^4x} - \frac{540.59}{\log^5x}} \tag{3.5}
\label{3.5}
\end{equation}
holds for every $x \geq 32$. The proofs of the remaining inequalities are similar to the proof of \eqref{3.5} and we leave the details to the reader. For every
$x \geq 5.5 \cdot 10^{25}$, Theorem \ref{thm302} implies the validity of \eqref{3.5}. Denoting the right-hand side of \eqref{3.5} by $f(x)$, we set $g(x) = 
f(x) - \text{li}(x) - \sqrt{x} \log x/(8\pi)$. We compute that $g(10^{14}) > 10^6$ and $g'(x) \geq 0$ for every $x \geq 10^{14}$. Hence $f(x) \geq \text{li}(x) 
+ \sqrt{x} \log x/(8 \pi)$ for every $x \geq 10^{14}$. Now we apply Proposition \ref{prop301} to get that the inequality \eqref{3.5} also holds for every $x$ 
such that $10^{14} \leq x \leq 5.5 \cdot 10^{25}$ as well. A comparsion with $\text{li}(x)$ shows that $f(x) > \text{li}(x)$ for every $x \geq 4\,560\,187$. 
From \eqref{3.4} follows that the inequality \eqref{3.5} also holds for every $x$ such that $4\,560\,187 \leq x \leq 10^{14}$. To verify that $f(x) > \pi(x)$ 
holds for every $x$ such that $67 \leq x \leq 4\,560\,187$, it suffices to 
check that $f(p_n) > \pi(p_n)$ for every positive integer $n$ such that $\pi(67) \leq n \leq \pi(4\,560\,187) + 1$, since $f(x)$ is strictly increasing for 
every $x \geq 67$. We conclude by direct computation. 
\end{proof}

In \cite[Theorem 1.3]{axler2016}, the present author purports that the inequality 
\begin{equation}
\pi(x) < \frac{x}{\log x - 1 - \frac{1}{\log x} - \frac{3.35}{\log^2x} - \frac{12.65}{\log^3x} - \frac{71.7}{\log^4x} - \frac{466.1275}{\log^5x} - 
\frac{3489.8225}{\log^6x}} \tag{3.6} \label{3.6}
\end{equation}
holds for every $x \geq e^{3.804}$. But the proof of this inequality in its present form is not correct. There is a mistake in the first part of the proof, 
where it is claimed that the inequality \eqref{3.6} holds for every $x \geq 10^{14}$.
Fortunaly, this incorrectness will be fixed by Theorem \ref{thm302}.  

\begin{kor} \label{kor304}
For every $x \geq e^{3.804}$, we have
\begin{displaymath}
\pi(x) < \frac{x}{\log x - 1 - \frac{1}{\log x} - \frac{3.35}{\log^2x} - \frac{12.65}{\log^3x} - \frac{71.7}{\log^4x} - \frac{466.1275}{\log^5x} - 
\frac{3489.8225}{\log^6x}}. 
\end{displaymath}
\end{kor}

\begin{proof}
The proof in \cite{axler2016} that the inequality \eqref{3.6} holds for every $x$ such that $e^{3.804} \leq x \leq 10^{14}$ is still correct and it suffices to
consider the remaining case $x \geq 10^{14}$. In this case the required inequality follows directly from Theorem \ref{thm302}.
\end{proof}

Using Proposition \ref{prop205}, we get the following upper bound for the prime counting function, which improve the inequality \eqref{3.3} for every 
sufficiently large values of $x$.

\begin{prop} \label{prop305}
For every $x \geq 41$, we have
\begin{displaymath}
\pi(x) < \frac{x}{\log x - 1 - \frac{1}{\log x} - \frac{3}{\log^2x} - \frac{113}{\log^3 x}}. \tag{3.7} \label{3.7}
\end{displaymath}
\end{prop}

\begin{proof}
The proof is similar to the proof of Theorem \ref{thm302} and we leave the details to the reader. We denote the right-hand side of \eqref{3.7} by $f(x)$ and 
let $x_1= 10^{15}$. Comparing $f(x)$ with $J_{4, 100, x_1}(x)$, we get, by using 
$f(x) > J_{4, 100, x_1}(x)$ holds for every $x \geq 10^{15}$. Then, by 
\eqref{3.2} and Proposition \ref{prop205}, that
$f(x) > \pi(x)$ for every $x \geq 10^{15}$. Next, we compare $f(x)$ with $\text{li}(x)$ and obtain that the desired inequality holds for every $x$ such that 
$e^7 \leq x \leq 10^{15}$ as well. A direct computation for smaller values of $x$ completes the proof.
\end{proof}

Integration of parts in \eqref{1.3} implies that for every positive integer $m$, we have
\begin{equation}
\pi(x) = \frac{x}{\log x} + \frac{x}{\log^2 x} + \frac{2x}{\log^3 x} + \frac{6x}{\log^4 x} + \frac{24x}{\log^5 x} + \ldots + \frac{(m-1)! x}{\log^mx} + O \left( 
\frac{x}{\log^{m+1} x} \right). \tag{3.8} \label{3.8}
\end{equation}
In this direction, we get the following upper bound for the prime counting function.

\begin{prop} \label{prop306}
For every $x > 1$, we have
\begin{equation}
\pi(x) < \frac{x}{\log x} + \frac{x}{\log^2 x} + \frac{2x}{\log^3 x} + \frac{6.15x}{\log^4 x} + \frac{24.15x}{\log^5 x} + \frac{120.75x}{\log^6 x} + 
\frac{724.5x}{\log^7 x} + \frac{6601x}{\log^8 x}. \tag{3.9} \label{3.9}
\end{equation} 
\end{prop}

\begin{proof}
We set $x_1 = 10^{15}$. Further, let $f(x)$ be the right-hand side of \eqref{3.9}. A comparsion with $J_{3, 0.15, x_1}(x)$ shows that $f(x) > J_{3, 0.15, 
x_1}(x)$ for every $x \geq 10^{15}$. By \eqref{3.2} and Theorem \ref{thm204}, we get that $f(x) > \pi(x)$ for every $x \geq x_1$. Next, we compare $f(x)$ with 
$\text{li}(x)$ and get that $f(x) > \text{li}(x)$ for every $x \geq 1\,509\,412$. Together with \eqref{3.4}, we obtain that $f(x) > \pi(x)$ for every $x$ such 
that $1\,509\,412 \leq x \leq 10^{15}$ as well. It remains to deal with the case where $1 < x \leq 1\,509\,412$. Since $f(x)$ is a strictly increasing function 
for every $x \geq 47$, it suffices to check that $f(p_n) > \pi(p_n)$ for every positive integer $n$ such that $\pi(47) \leq n \leq \pi(1\,509\,412) + 1$. For 
smaller values of $x$, we conclude by direct computation.
\end{proof}

\begin{rema}
Using Proposition \ref{prop205}, instead of Theorem \ref{thm204}, in the proof of Proposition \ref{prop306}, we get similarly that the inequality
\begin{displaymath}
\pi(x) < \frac{x}{\log x} + \frac{x}{\log^2 x} + \frac{2x}{\log^3 x} + \frac{6x}{\log^4 x} + \frac{133x}{\log^5 x}
\end{displaymath}
holds for every $x > 1$.
\end{rema}


We get the following weaker but more compact upper bound for the prime counting function.

\begin{kor} \label{kor307}
For every $x \geq 27\,777\,762\,891$, we have
\begin{displaymath}
\pi(x) < \frac{x}{\log x} + \frac{x}{\log^2 x} + \frac{2.3x}{\log^3 x}.
\end{displaymath} 
\end{kor}

\begin{proof}
From Proposition \ref{prop306} follows that the required inequality holds for every $x \geq 5.1 \cdot 10^{10}$. Denoting the right-hand side of the desired
inequality by $f(x)$, we get that $f(x) > \text{li}(x)$ for every $x \geq 33\,272\,003\,003$. Together with \eqref{3.4}, we conclude the proof for every $x 
\geq 33\,272\,003\,003$. For every positive integer $n$ such that $\pi(27\,777\,762\,917) \leq n \leq \pi(33\,272\,003\,003)$, we check that $f(p_n) \geq 
\pi(p_n)$. Since $f$ is an increasing function for every $x \geq 7$, we get that $f(x) > \pi(x)$ for every $x$ such that $ 27\,777\,762\,917 \leq x < 
33\,272\,003\,003$. A direct computer check for small values of $x$ completes the proof.
\end{proof}




\subsection{New lower bounds for the prime counting function}

In this subsection, we give new lower bounds for the prime counting function, which improve the currently best known lower bound given in \cite[Theorem
1.4]{axler2016}, namely
\begin{displaymath}
\pi(x) > \frac{x}{\log x - 1 - \frac{1}{\log x} - \frac{2.65}{\log^2x} - \frac{13.35}{\log^3x} - \frac{70.3}{\log^4x} - \frac{455.6275}{\log^5x} - 
\frac{3404.4225}{\log^6x}}
\end{displaymath}
for every $x \geq 1\,332\,479\,531$.

\begin{thm} \label{thm308}
For every $x \geq 19\,033\,744\,403$, we have
\begin{equation}
\pi(x) > \frac{x}{\log x - 1 - \frac{1}{\log x} - \frac{2.85}{\log^2x} - \frac{13.15}{\log^3x} - \frac{70.7}{\log^4x} - \frac{458.7275}{\log^5x} - 
\frac{3428.7225}{\log^6x}}. \tag{3.10} \label{3.10}
\end{equation}
\end{thm}

\begin{proof}
Let $x_1 = 5 \cdot 10^9$. Further, let $f(x)$ be the right-hand side of \eqref{3.10} and let $r(x)$ be the denominator of $f(x)$. To prove that the function 
$g(x) = J_{3, -0.15, x_1}(x) - f(x)$ is positive for every $x \geq x_1$, we need to show that $g(x_1) > 0$ and that the derivative of $g$ is positive for every 
$x \geq x_1$. By Dusart \cite[Table 6.2]{dusart2010}, we have $\vartheta(x_1) \leq 4\,999\,906\,576$. Combined with $\pi(x_1) = 234\,954\,223$, we compute that 
$g(x_1) > 18.955$.
To show that the derivative of $g$ is positive for every $x \geq x_1$, we set 
\begin{align*}
h(y) & = 28\,930y^{10} + 11\,393y^9 + 37\,131y^8 + 151\,211y^7 + 697\,310y^6 + 3\,145\,306y^5 + 11\,749\,355y^4\\
& \p{\q\q} - 34\,521y^3 - 158\,992y^2 - 347\,857y + 5\,290\,262.
\end{align*}
Clearly, we have $h(y) > 0$ for every $y \geq \log(x_1)$. Hence, $g'(x) r^2(x) \log^{19}x \geq h_1(\log x) \geq 0$ for every $x \geq x_1$. So, $J_{3, -0.15, 
x_1}(x) - f(x) = g(x) > 0$ for every $x \geq x_1$. Using \eqref{3.2} and Theorem \ref{thm204}, we get that required inequality for every $x \geq 
19\,035\,709\,163$. To deal with the remaining case where $19\,033\,744\,403 \leq x \leq 19\,035\,709\,163$, we note that $f(x)$ is increasing for every $x \geq 
91$. So we check with a computer that $\pi(p_n) > f(p_{n+1})$ for every positive integer $n$ such that $\pi(19\,033\,744\,403) \leq n \leq 
\pi(19\,035\,709\,163)$.
\end{proof}

In the next corollary, we establish some weaker lower bounds for the prime counting function.

\begin{kor} \label{kor309}
We have
\begin{displaymath}
\pi(x) > \frac{x}{\log x - 1 - \frac{a_1}{\log x} - \frac{a_2}{\log^2x} - \frac{a_3}{\log^3x} - \frac{a_4}{\log^4x} - \frac{a_5}{\log^5x}}
\end{displaymath}
for every $x \geq x_0$, where
\begin{center}
\begin{tabular}{|l||c|c|c|c|c|c|c|c|c|c|}
\hline
$a_1$\rule{0mm}{4mm} & $                 1$ & $               1$ & $                 1$ & $1$ & $1$\\ \hline
$a_2$\rule{0mm}{4mm} & $             2.85$ & $          2.85$ & $            2.85$ & $2.85$ & $0$\\ \hline
$a_3$\rule{0mm}{4mm} & $           13.15$ & $         13.15$ & $          13.15$ & $0$ & $0$\\ \hline 
$a_4$\rule{0mm}{4mm} & $             70.7$ & $           70.7$ & $                0$ & $0$ & $0$\\ \hline
$a_5$\rule{0mm}{4mm} & $      458.7275$ & $               0$ & $                0$ & $0$ & $0$\\ \hline
$x_0$\rule{0mm}{4mm} & $11\,532\,441\,449$ & $7\,822\,207\,951$ & $1\,331\,532\,233$ & $38\,099\,531$  & $468\,049$\\ \hline
\end{tabular} \ .
\end{center}
\end{kor}

\begin{proof}
By comparing each right-hand side with the right-hand side of \eqref{3.10}, we see that each inequality holds for every $x \geq 19\,033\,744\,403$. For smaller
values of $x$ we use computer.
\end{proof}

Now, we apply Proposition \ref{prop205} to obtain the following result, which refines Theorem \ref{thm308} for all sufficiently large values of $x$.

\begin{prop} \label{prop310}
For every $x \geq 19\,423$, we have
\begin{equation}
\pi(x) > \frac{x}{\log x - 1 - \frac{1}{\log x} - \frac{3}{\log^2x} + \frac{87}{\log^3 x}}. \tag{3.11} \label{3.11}
\end{equation}
\end{prop}

\begin{proof}
Let $x_1 = 10^6$ and denote the right-hand side of \eqref{3.11} by $f(x)$. A comparsion with $J_{4, -100, x_1}(x)$ gives that  $J_{4, -100, x_1}(x) > f(x)$ for
every $x \geq 10^6$. Now we use \eqref{3.2} and Proposition \ref{prop205} to get that $\pi(x) > f(x)$ for every $x \geq 10^6$. To prove that the inequality 
\eqref{3.11} is also valid for every $x$ such that $19423 \leq x < 10^6$, it suffices to check with a computer that $\pi(p_n) > f(p_{n+1})$ for every positive 
integer $n$ such that $\pi(19\,423) \leq n \leq \pi(10^6)$, since $f$ is a strictly increasing function on the interval $(1, \infty)$.
\end{proof}

The asymptotic expansion \eqref{3.8} implies that the inequality
\begin{displaymath}
\pi(x) > \frac{x}{\log x} + \frac{x}{\log^2 x} + \frac{2x}{\log^3 x} + \frac{6x}{\log^4 x} + \frac{24x}{\log^5 x} + \ldots + \frac{(n-1)! x}{\log^nx}
\end{displaymath}
holds for all sufficiently large values of $x$. The best explicit result in this direction was given in \cite[Theorem 1.2]{axler2016}, namely that
\begin{equation}
\pi(x) > \frac{x}{\log x} + \frac{x}{\log^2 x} + \frac{2x}{\log^3 x} + \frac{5.65x}{\log^4 x} + \frac{23.65x}{\log^5 x} + \frac{118.25x}{\log^6 x} + 
\frac{709.5x}{\log^7 x} + \frac{4966.5x}{\log^8 x} \tag{3.12} \label{3.12}
\end{equation} 
for every $x \geq 1\,332\,450\,001$. A consequence of Theorem \ref{thm308} is the following refinement of \eqref{3.12}.

\begin{prop} \label{prop311}
For every $x \geq 19\,027\,490\,297$, we have
\begin{displaymath}
\pi(x) > \frac{x}{\log x} + \frac{x}{\log^2 x} + \frac{2x}{\log^3 x} + \frac{5.85x}{\log^4 x} + \frac{23.85x}{\log^5 x} + \frac{119.25x}{\log^6 x} + 
\frac{715.5x}{\log^7 x} + \frac{5008.5x}{\log^8 x}.
\end{displaymath}
\end{prop}

\begin{proof}
Let $U(x)$ denotes the right-hand side of the required inequality
and let $R(y) = U(y)\log y/y$. Further, we set $S(y) = (y^7 - y^6 - y^5 - 2.85y^4 - 13.15y^3 - 70.7y^2 - 458.7275y - 3428.7225)/y^6$. Then $S(y) > 0$ for every
$t > 3.79$ and $y^{13}R(y)S(y) = y^{14} - T(y)$, where $T(y) = 11137.2625y^6 + 19843.008375y^5 + 63112.7025y^4 + 252925.911y^3 + 1091195.634375y^2 + 
475078.76325y + 17172756.64125$. By Theorem \eqref{thm308},
\begin{displaymath}
\pi(x) > \frac{x}{S(\log x)} > \frac{x}{S(\log x)} \left( 1 - \frac{T(\log x)}{\log^{14} x} \right) = U(x)
\end{displaymath}
for every $x \geq 19\,033\,744\,403$. So it remains to deal with the case where $19\,027\,490\,297 \leq x \leq 19\,033\,744\,403$. Since $U(x)$ is a strictly
increasing function for every $x \geq 44$, it suffices to check with a computer that $\pi(p_n) > U(p_{n+1})$ for every positive integer $n$ such that 
$\pi(19\,027\,490\,297) \leq n \leq \pi(19\,033\,744\,403)$.
\end{proof}

\section{On the existence of prime numbers in short intervals}

Bertrand's postulate states that for each positive integer $n$ there is a prime number $p$ with $n < p \leq 2n$, and was proved, for instance, by Chebyshev
\cite{cheby1850} and by Erdös \cite{er2}. In the following, we note some of the remarkable improvements of Bertrand's postulate. The first result is due to 
Schoenfeld \cite[Theorem 12]{schoenfeld1976}. He found that for every $x \geq 2\,010\,759.9$ there is a prime number $p$ with $x < p < x(1 + 1/16\,597)$. In 
2003, Ramar\'{e} and Saouter \cite[Theorem 3]{rasa}found that for every $x \geq 10\,726\,905\,041$ there is a prime number $p$ so that $x < p \leq x( 1 + 
1/28\,313\,999)$. Further, they \cite[Table 1]{rasa} gave a table of sharper results, which hold for large $x$. In 2014, Kadiri and Lumley \cite[Table 2]{kl} 
found a series of improvements. For instance, they showed that for every $x \geq e^{150}$ there is a prime number $p$ such that $x < p < x(1 + 
1/2\,442\,159\,713)$. In 1998, Dusart \cite[Th\'{e}or\`{e}me 1.9]{pd1} proved that for every $x \geq 3\,275$ there exists a prime number $p$ such that $x < p 
\leq x(1 + 1/(2 \log^2 x))$ and then, in 2010, reduced the interval himself \cite[Proposition 6.8]{dusart2010} by showing that for every $x \geq 396\,738$ there 
is a prime number $p$ satisfying $x < p \leq x(1+1/(25 \log^2 x))$. In 2016, Trudgian \cite[Corollary 2]{trud} proved that for every $x \geq 2\,898\,242$ there 
exists a prime number $p$ with
\begin{equation}
x < p \leq x\left( 1 + \frac{1}{111 \log^2 x} \right). \tag{4.1} \label{4.1}
\end{equation}
Recently, Dusart \cite[Corollary 5.5]{dusart2017} improved Trudgian's result by showing that for every $x \geq 468\,991\,632$ there exists a prime number $p$
such that
\begin{equation}
x < p \leq x\left( 1 + \frac{1}{5\,000 \log^2 x} \right). \tag{4.2} \label{4.2}
\end{equation}
In \cite[Theorem 1.5]{axler2016}, it is shown that for every $x \geq 58\,837$ there is a prime number $p$ such that $x < p \leq x( 1 + 1.1817/\log^3 x)$. In
\cite[Proposition 5.4]{dusart2017}, Dusart refined the last result by showing that for every $x \geq 89\,693$ there exists a prime number $p$ such that
\begin{equation}
x < p \leq x\left( 1 + \frac{1}{\log^3 x} \right). \tag{4.3} \label{4.3}
\end{equation}
In the following theorem, we improve \eqref{4.3} on the one hand by decreasing the coefficient of the $1/\log^3 x$ term and on the other hand by increasing the
exponent of the $\log x$ term. In order to do this, we use some estimates for the prime counting function obtained in Section 3.

\begin{thm} \label{thm401}
For every $x \geq 6\,034\,256$ there is a prime number $p$, such that
\begin{displaymath}
x < p \leq x \left( 1 + \frac{0.087}{\log^3x} \right),
\end{displaymath}
and for every $x > 1$ there is a prime number $p$, such that
\begin{equation}
x < p \leq x \left( 1 + \frac{198.2}{\log^4x} \right). \tag{4.4} \label{4.4}
\end{equation}
\end{thm}

\begin{proof}
Similar to the proof of Theorem \ref{thm204}, we get that
\begin{equation}
|\vartheta(x) - x| < \frac{0.043x}{\log^3x} \tag{4.5} \label{4.5}
\end{equation}
for every $x \geq e^{40}$. Setting $f(x) = 0.087/\log^3x$, we use \eqref{4.5} to get that
\begin{displaymath}
\vartheta(x+xf(x)) - \vartheta(x) > \frac{x}{\log^3x} \left( 0.001 - \frac{0.003741}{\log^3x} \right) \geq 0
\end{displaymath}
for every $x \geq e^{40}$, which implies that for every $x \geq e^{40}$ there is a prime number $p$ satisfying $x < p \leq x(1+0.087/\log^3x)$. From \eqref{4.2} 
it is clear that the claim follows for every $x$ such that $468\,991\,632 \leq x \leq e^{40}$. To deal with the case where $156\,007 \leq x \leq 468\,991\,632$, 
we check with a computer that the inequality $p_n(1 + 0.087/\log^3 p_n) > p_{n+1}$ holds for every positive integer $n$ such that $\pi(6\,034\,393) \leq n \leq 
\pi(468\,991\,632)$. Finally, we notice that $\pi(x(1+ 0.087/\log^3 x)) > \pi(x)$ for every $x$ such that $6\,034\,256 \leq x < 6\,034\,393$, which completes 
the proof of the first part.

We define $g(x) = 198.2/\log^4x$. To show the second part, we first note that
\begin{equation}
|\vartheta(x) - x| < \frac{99.07x}{\log^4x} \tag{4.6} \label{4.6}
\end{equation}
for every $x \geq e^{25}$. The proof of this inequality is quite similar to the proof of Proposition \ref{prop205} and we leave the details to the reader. 
Using \eqref{4.6}, we obtain that
\begin{displaymath}
\vartheta(x+xg(x)) - \vartheta(x) > \frac{x}{\log^4x} \left( 0.06 - \frac{19635.674}{\log^4x} \right) \geq 0
\end{displaymath}
for every $x \geq e^{25}$. Analogously to the proof of the first part, we check with a computer that for every $1 < x < e^{25}$ there is a prime $p$ so that $x 
< p \leq x( 1 + 198.2/\log^4x)$.
\end{proof}

By using \eqref{4.1}, Dudek \cite[Theorem 3.6]{dudek} purports to prove that for every positive integer $m \geq 4.971 \cdot 10^9$ there exists a prime number
between $n^m$ and $(n+1)^m$ for all $n \geq 1$. In fact, he showed the slightly weaker lower bound $m \geq 4.97117 \cdot 10^9$. Applying \eqref{4.4} to Dudek's 
proof, we get the following refinement.

\begin{prop} \label{prop402}
Let $m \geq 3\,239\,773\,013$. Then there is a prime between $n^m$ and $(n+1)^m$ for all $n \geq 1$. 
\end{prop}

\begin{proof}
Let $m \geq M_0$, where $M_0 = 3\,239\,773\,013$. First, we set $x = n^m$ in \eqref{4.4} to get that there is a prime $p$ so that
\begin{equation}
n^m \leq p < n^m \left( 1 + \frac{198.2}{\log^4(n^m)} \right) \tag{4.7} \label{4.7}
\end{equation}
for every $n \geq 2$. We have
\begin{equation}
n^m \left( 1 + \frac{198.2}{\log^4(n^m)} \right) \leq n^m + mn^{m-1} \tag{4.8} \label{4.8}
\end{equation}
if and only if $198.2n/\log^4n \leq m^5$. Setting $n_0(m) = \max \{ x \in \N \mid 198.2x/\log^4x \leq m^5 \}$, we get $n_0(m) \geq n_0(M_0) \geq 4.18498732 
\cdot 10^{53}$.
Now, we apply \eqref{4.8} to \eqref{4.7} to get that there is a prime $p$ so that
\begin{equation}
n^m \leq p < n^m + mn^{m-1} \tag{4.9} \label{4.9}
\end{equation}
for every $2 \leq n \leq n_0(m)$. By the binomial theorem, we have $n^m + mn^{m-1} \leq (n+1)^m$. So, \eqref{4.9} implies that there is a prime between $n^m$
and $(n+1)^m$ for every $2 \leq n \leq n_0(m)$. On the other hand, Dudek \cite[p. 42]{dudek} showed that for every positive integer $t \geq 1000$ there is a 
prime between $n^t$ and $(n+1)^t$ for every $n \geq n_1(t)$, where $n_1(t) = \exp(1000\exp(19.807)/t)$. Therefore
\begin{displaymath}
n_1(m) = \exp \left( \frac{1000 \exp(19.807)}{m} \right) \leq \exp \left( \frac{1000 \exp(19.807)}{M_0} \right) \leq 4.1849871 \cdot 10^{53}.
\end{displaymath}
Since $n_1(m) \leq n_0(m)$, we conclude the proof for every $n \geq 2$. The remaining case $n=1$ is clear.
\end{proof}

\section{On estimates of two sums over primes}

In this section, we give some refined estimates for the sums
\begin{displaymath}
\sum_{p \leq x} \frac{1}{p}, \q \sum_{p \leq x} \frac{\log p}{p},
\end{displaymath}
where $p$ runs over primes not exceeding $x$.

\subsection{On the sum of the reciprocals of all prime numbers not exceeding $x$}

In 1737, Euler \cite{euler1737} proved that the sum of the reciprocals of all prime numbers diverges.
In particular, this result implies that there are infinitely many primes. Further, Euler \cite[Theorema 19]{euler1737} and later Gauss \cite{gauss1791} stated
that the sum of the reciprocals of all prime numbers not exceeding $x$ grows like $\log \log x$. In 1874, Mertens \cite[p. 52]{mertens1874} used several 
results of Chebyshev's papers \cite{cheby1848}, \cite{cheby1850} to find  that $\log \log x$ is the right order of magnitude for the sum of the reciprocals of 
all prime numbers $p$ not exceeding $x$ by showing that 
\begin{equation}
\sum_{p \leq x} \frac{1}{p} = \log \log x + B + O \left( \frac{1}{\log x} \right). \tag{5.1} \label{5.1}
\end{equation}
Here, $B$ denotes the Mertens' constant (see \cite{sloane1}) and is defined by
\begin{equation}
B = \gamma + \sum_{p} \left( \log \left( 1 - \frac{1}{p} \right) + \frac{1}{p} \right) = 0.2614972128476427837554268386 \ldots , \tag{5.2} \label{5.2}
\end{equation}
where $\gamma = 0.577215664901532860606512090082402431\ldots$ denotes the Euler-Mascheroni constant. In 1962, Rosser and Schoenfeld \cite[p. 74]{rosser1962}
derived a remarkable identity, which connects the sum of the reciprocals of all prime numbers not exceeding $x$ with Chebyshev's $\vartheta$-function, namely
\begin{equation}
\sum_{p \leq x} \frac{1}{p} = \log \log x + B + \frac{\vartheta(x) - x}{x \log x} - \int_x^{\infty} \frac{(\vartheta(y)-y)(1 + \log y)}{y^2\log^2y} \, dy. 
\tag{5.3} \label{5.3}
\end{equation}
Together with \eqref{2.4}, they \cite[p. 68]{rosser1962} refined the error term in Mertens' result \eqref{5.1} by giving
\begin{displaymath}
\sum_{p \leq x} \frac{1}{p} = \log \log x + B + O (\exp(-a\sqrt{\log x})).
\end{displaymath}
Using \eqref{5.3} and explicit estimates for Chebyshev's $\vartheta$-function, Rosser and Schoenfeld \cite[Theorem 5]{rosser1962} were able to find
\begin{displaymath}
\log \log x + B - \frac{1}{2 \log^2 x} < \sum_{p \leq x} \frac{1}{p} < \log \log x + B + \frac{1}{2 \log^2 x},
\end{displaymath}
where the left-hand side inequality is valid for every $x > 1$ and the right-hand side inequality holds for every $x \geq 286$. After some remarkable 
improvements, the currently best known estimates for the sum of the reciprocals of all prime numbers not exceeding $x$ are due to Dusart \cite[Theorem 
5.6]{dusart2017}. He used \eqref{5.3} together with \eqref{2.5} to get that
\begin{equation}
\left| \sum_{p \leq x} \frac{1}{p} - \log \log x - B \right| \leq \frac{\eta_k}{k\log^kx} + \frac{(k+2)\eta_k}{(k+1)\log^{k+1}x} \tag{5.4} \label{5.4}
\end{equation}
for every $x \geq x_0(k)$. Then he \cite[Theorem 5.6]{dusart2017} applied Lemma \ref{lem203} with $k = 3$ and $\eta_3 = 0.5$, and get
\begin{equation}
- \frac{1}{5\log^3x} \leq \sum_{p \leq x} \frac{1}{p} - \log \log x - B  \leq \frac{1}{5\log^3x} \tag{5.5} \label{5.5}
\end{equation}
for every $x \geq 2\,278\,383$. Following Dusart's proof of \eqref{5.5}, we obtain the following slightly refinements of these estimates by using Theorem 
\ref{thm204}.

\begin{prop} \label{prop501}
We have
\begin{displaymath}
- \frac{1}{20\log^3x} - \frac{3}{16\log^4x} \leq \sum_{p \leq x} \frac{1}{p} - \log \log x - B \leq \frac{1}{20\log^3x} + \frac{3}{16\log^4x},
\end{displaymath}
where the left-hand side inequality holds for every $x > 1$ and the right-hand side inequality is valid for every $x \geq 46\,909\,074$.
\end{prop}

\begin{proof}
We use \eqref{5.4} and Theorem \ref{thm204} to get that these inequalities hold for every $x \geq 19\,035\,709\,163$. To verify that the left-hand side
inequality holds for every $x$ such that $2 \leq x \leq 19\,035\,709\,163$ as well, we check with a computer that for every positive integer $n \leq 
\pi(19\,035\,709\,163)$,
\begin{displaymath}
\sum_{k\leq n} \frac{1}{p_k} \geq \log \log p_{n+1} + B - \frac{1}{20\log^3 p_{n+1}} - \frac{3}{16\log^4 p_{n+1}}.
\end{displaymath}
Clearly, the left-hand side inequality holds for every $x$ such that $1 < x < 2$. A similar calculation shows that the right-hand side inequality holds for
every $x$ such that $46\,909\,074 \leq x \leq 19\,035\,709\,163$ as well.
\end{proof}



\subsection{On another sum over all prime numbers not exceeding $x$}

In 1857, de Polignac \cite[part 3]{polignac} stated without proof that $\log x$ is the right asymptotic behaviour for the sum 
\begin{equation}
\sum_{p \leq x} \frac{\log p}{p}, \tag{5.6} \label{5.6}
\end{equation}
where $p$ runs over primes not exceeding $x$. A rigorous proof for this was given by Mertens \cite[p. 49]{mertens1874} in 1874. He showed that
\begin{equation}
\sum_{p \leq x} \frac{\log p}{p} = \log x + O(1). \tag{5.7} \label{5.7}
\end{equation}
In 1909, Landau \cite[§55]{landau} was able to precise \eqref{5.7} by finding
\begin{displaymath}
\sum_{p \leq x} \frac{\log p}{p} = \log x + E + O (\exp(-\sqrt[14]{\log x})),
\end{displaymath}
where $E$ is a constant (see \cite{sloane2}) defined by
\begin{displaymath}
E = - \gamma - \sum_{p} \frac{\log p}{p(p-1)} = -1.332582275733220881765828776071027748838459 \ldots.
\end{displaymath}
Rosser and Schoenfeld \cite[p. 74]{rosser1962} connected the sum in \eqref{5.6} with Chebyshev's $\vartheta$-function by showing
\begin{equation}
\sum_{p \leq x} \frac{\log p}{p} = \log x + E + \frac{\vartheta(x) - x}{x} - \int_x^{\infty} \frac{\vartheta(y) - y}{y^2} \, dy. \tag{5.8} \label{5.8}
\end{equation}
Using their own explicit estimates for Chebychev's $\vartheta$-function, they \cite[Theorem 6]{rosser1962} found
\begin{displaymath}
\log x + E - \frac{1}{2 \log x} < \sum_{p \leq x} \frac{\log p}{p} < \log x + E + \frac{1}{2\log x},
\end{displaymath}
where the left-hand side inequality is valid for every $x > 1$ and the right-hand side inequality holds for every $x \geq 319$. In 2010, Dusart \cite[Theorem
6.11]{dusart2010} utilized \eqref{5.8} and \ref{2.5} to get that the inequality
\begin{equation}
\left| \sum_{p \leq x} \frac{\log p}{p} - \log x - E \right| \leq \frac{\eta_k}{(k-1)\log^{k-1}x} + \frac{\eta_k}{\log^kx} \tag{5.9} \label{5.9}
\end{equation}
holds for every $x \geq x_0(k)$. Then he \cite[Theorem 5.7]{dusart2016} applied Lemma \ref{lem203} with $k = 3$ and $\eta_3 = 0.5$ to \eqref{5.9} and obtained
the current best estimates for the sum given in \eqref{5.6}, namely
\begin{displaymath}
- \frac{0.3}{\log^2 x} < \sum_{p \leq x} \frac{\log p}{p} - \log x - E < \frac{0.3}{\log^2 x}
\end{displaymath}
for every $x \geq 912\,560$. Now, \eqref{5.9} and Theorem \ref{thm204} imply the following refinement.

\begin{prop} \label{prop504}
We have
\begin{equation}
- \frac{3}{40\log^2x} - \frac{3}{20\log^3x} \leq \sum_{p \leq x} \frac{\log p}{p} - \log x - E  \leq \frac{3}{40\log^2x} + \frac{3}{20\log^3x}, \tag{5.10}
\label{5.10}
\end{equation}
where the left-hand side inequality is valid for every $x > 1$ and the right-hand side inequality holds for every $x \geq 30\,972\,320$.
\end{prop}

\begin{proof}
From \eqref{5.9} and Theorem \ref{thm204}, it follows that the required inequalities \eqref{5.10} holds for every $x \geq 19\,035\,709\,163$. Similarly to the
proof of Proposition \ref{prop501}, we use a computer to check the desired inequalities for smaller values of $x$.
\end{proof}



\section{Refined estimates for a product over primes}

The asymptotic formula \eqref{5.1} implies that
\begin{displaymath}
\prod_{p \leq x} \left( 1 - \frac{1}{p} \right) = \frac{e^{-\gamma}}{\log x} + O \left( \frac{1}{\log^2 x} \right)
\end{displaymath}
and in this direction, Rosser and Schoenfeld \cite[Theorem 7]{rosser1962} found that
\begin{equation}
\frac{e^{-\gamma}}{\log x} \left( 1 - \frac{1}{2 \log^2 x} \right) < \prod_{p \leq x} \left( 1 - \frac{1}{p} \right) < \frac{e^{-\gamma}}{\log x} \left( 1 + 
\frac{1}{2 \log^2 x} \right), \tag{6.1} \label{6.1}
\end{equation}
where the left-hand side inequality is valid for every $x \geq 285$ and the right-hand side inequality holds for every $x > 1$. After several improvements, the
sharpest known estimates for this product are due to Dusart \cite[Theorem 5.9]{dusart2016}. Following Rosser's and Schoenfeld's proof of \eqref{6.1}, Dusart 
used \eqref{5.4} and Lemma \ref{lem203} with $k=3$ and $\eta_k = 0.5$ to find
\begin{displaymath}
\frac{e^{-\gamma}}{\log x} \left( 1 - \frac{0.2}{\log^3 x} \right) < \prod_{p \leq x} \left( 1 - \frac{1}{p} \right) < \frac{e^{-\gamma}}{\log x} \left( 1 +
\frac{0.2}{\log^3 x} \right)
\end{displaymath}
for every $x \geq 2\,278\,382$. We use the same method combined with Proposition \ref{prop501} to obtain the following

\begin{prop} \label{prop601}
For every $x \geq 46\,909\,038$, we have
\begin{equation}
\prod_{p \leq x} \left( 1 - \frac{1}{p} \right) > \frac{e^{-\gamma}}{\log x} \left( 1 - \frac{1}{20\log^3x} - \frac{3}{16\log^4x} \right), \tag{6.2} \label{6.2}
\end{equation}
and for every $x > 1$, we have
\begin{displaymath}
\prod_{p \leq x} \left( 1 - \frac{1}{p} \right) < \frac{e^{-\gamma}}{\log x} \left( 1 + \frac{0.07}{\log^3x} \right).
\end{displaymath}
\end{prop}

\begin{proof}
First, let $x \geq 46\,909\,074$ and $S = \sum_{p>x} ( \log (1 - 1/p) + 1/p) = - \sum_{k=2}^{\infty} \sum_{p > x} 1/kp^k.$
Using the right-hand side inequality of Proposition \ref{prop501} and the definition \eqref{5.2} of $B$, we have
\begin{displaymath}
\prod_{p\leq x} \left( 1 - \frac{1}{p} \right) > \frac{e^{-\gamma}}{\log x} \exp \left( -S - \frac{1}{20\log^3x} - \frac{3}{16\log^4x} \right).
\end{displaymath}
Now we use the inequality $e^t \geq 1+t$, which holds for every real $t$, and the fact that $S < 0$ to obtain that the required inequality \eqref{6.2} holds
for every $x \geq 46\,909\,074$. We conclude with a computer check.

Analogously, we use the left-hand side inequality of Proposition \ref{prop501} to get
\begin{displaymath}
\prod_{p\leq x} \left( 1 - \frac{1}{p} \right) < \frac{e^{-\gamma}}{\log x} \exp \left( - S + \frac{1}{20\log^3x} + \frac{3}{16\log^4x} \right)
\end{displaymath}
for every $x > 1$. By Rosser and Schoenfeld \cite[p. 87]{rosser1962}, we have $-S < 1.02/((x-1)\log x)$ for every $x > 1$. Since
\begin{displaymath}
\frac{1}{20\log^3x} + \frac{3}{16\log^4x} + \frac{1.02}{(x-1)\log x} \leq \frac{0.06}{\log^3x} \leq \log \left( 1 + \frac{0.07}{\log^3x} \right)
\end{displaymath}
for every $x \geq 1.4 \cdot 10^8$, we get that the desired upper bound holds for every $x \geq 1.4 \cdot 10^8$. We conclude by direct computation.
\end{proof}



\section*{Acknowledgement}
I would like to express my great appreciation to Jan Büthe and Marc Del\'{e}glise for the computation of several special values of Chebyshev's 
$\vartheta$-function. I also would like to thank Marco Aymone for pointing out a mistake in the proof of Proposition \ref{prop601}, which has been corrected in 
this version.




\begin{thebibliography}{10}
\bibitem{ap} T. Apostol, \emph{Introduction to analytic number theory}, Springer, New York--Heidelberg, 1976.
\bibitem{axler2016} C. Axler, \emph{New bounds for the prime counting function}, Integers \textbf{16} (2016), Paper No. A22, 15 pp.
\bibitem{buethe} J. Büthe, \emph{An analytic method for bounding $\psi(x)$}, to appear in Math. Comp.
\bibitem{buethepre2} J. Büthe, \emph{An improved analytic method for calculating $\pi(x)$}, Manuscripta Math. \textbf{151} (2016), no. 3--4, 329--352.
\bibitem{buethe2016} J. Büthe, \emph{Estimating $\pi(x)$ and related functions under partial RH assumptions}, Math. Comp. \textbf{85} (2016), no. 301, 
2483--2498.
\bibitem{cheby1848} P. L. Chebyshev, \emph{Sur la fonction qui d\'{e}termine la totalit\'{e} des nombres premiers inf\'{e}rieurs \`{a} une limite donn\'{e}e}, 
M\'{e}moires des savants \'{e}trangers de l'Acad. Sci. St.P\'{e}tersbourg \textbf{6} (1848), 1--19. [Also in J. math. pures appli. \textbf{17} (1852), 
341--365.]
\bibitem{cheby1850} P. L. Chebyshev, \emph{M\'{e}moire sur les nombres premiers}, M\'{e}moires des savants \'{e}trangers de l'Acad. Sci. St.P\'{e}tersbourg 
\textbf{7} (1850), 17--33. [Also in J. math. pures appl. \textbf{17} (1852), 366--390.]
\bibitem{cp} M. Cipolla, \emph{La determinazione assintotica dell' $n^{imo}$ numero primo}, Rend. Accad. Sci. Fis-Mat. Napoli (3) \textbf{8} (1902), 132--166.
\bibitem{dudek} A. Dudek, \emph{An explicit result for primes between cubes}, preprint, 2014. Available at \url{arxiv.org/1401.4233}.
\bibitem{pd1} P. Dusart, \emph{Autour de la fonction qui compte le nombre de nombres premiers}, Dissertation, Universit\'{e} de Limoges, 1998.
\bibitem{dusart2010} P. Dusart, \emph{Estimates of some functions over primes without R.H.}, preprint, 2010. Available at \url{arxiv.org/1002.0442}.
\bibitem{dusart2016} P. Dusart, \emph{Estimates of $\psi$, $\theta$ for large values of $x$ without the Riemann hypothesis}, Math. Comp. \textbf{85} (2016), 
no. 298, 875--888.
\bibitem{dusart2017} P. Dusart, \emph{Explicit estimates of some functions over primes}, to appear in Ramanujan J.
\bibitem{er2} P. Erdös, \emph{Beweis eines Satzes von Tschebyschef}, Acta Litt. Sci. Szeged \textbf{5} (1932), 194--198.
\bibitem{euler1737} L. Euler, \emph{Variae observationes circa series infinitas}, Comment. Acad. Sci. Petropol. \textbf{9} (1744), 160--188.
\bibitem{faberkadiri} L. Faber and H. Kadiri, \emph{New bounds for $\psi(x)$}, Math. Comp. \textbf{84} (2015), no. 293, 1339--1357. 
\bibitem{gauss1791} C. F. Gauß, \emph{Asymptotische Gesetze der Zahlentheorie}, Werke, \textbf{10.1}, Teubner, Leipzig, 1917, 11--16.
\bibitem{gauss} C. F. Gauß, \emph{Werke}, 2 ed., Königlichen Gesellschaft der Wissenschaften, Göttingen, 1876.
\bibitem{gourdon} X. Gourdon, private conversation.
\bibitem{hadamard1896} J. Hadamard, \emph{Sur la distribution des z\'{e}ros de la fonction $\zeta(s)$ et ses cons\'{e}quences arithm\'{e}tiques}, Bull. Soc. 
Math. France \textbf{24} (1896), 199--220. 
\bibitem{kl} H. Kadiri and A. Lumley, \emph{Short effective intervals containing primes}, Integers \textbf{14} (2014), Paper No. A61, 18 pp.
\bibitem{koch1901} H. von Koch, \emph{Sur la distribution des nombres premiers}, Acta Math. \textbf{24} (1901), no. 1, 159--182. 
\bibitem{korobov1958} N. M. Korobov, \emph{Estimates of trigonometric sums and their applications}, Uspehi Mat. Nauk \textbf{13} (1958), no. 4 (82), 185--192.
\bibitem{landau} E. Landau, \emph{Handbuch der Lehre yon der Verteilung der Primzahlen}, 2 vols., Leipzig, Teubner, 1909. Reprinted in 1953 by Chelsea 
Publishing Co., New York.
\bibitem{mr} J.-P. Massias and G. Robin, \emph{Bornes effectives pour certaines fonctions concernant les nombres premiers}, Journal Th. Nombres de Bordeaux, 
Vol. \textbf{8} (1996), 213--238.
\bibitem{mertens1874} F. Mertens, \emph{Ein Beitrag zur analytischen Zahlentheorie}, J. Reine Angew. Math. \textbf{78} (1874), 42--62.
\bibitem{nark} W. Narkiewicz, \emph{The development of prime number theory. From Euclid to Hardy and Littlewood}, Springer Monographs in Mathematics, 
Springer-Verlag, Berlin, 2000, xii+448 pp.
\bibitem{pan3} L. Panaitopol, \emph{A formula for $\pi(x)$ applied to a result of Koninck-Ivi\'{c}}, Nieuw Arch. Wiskd. (5) \textbf{1} (2000), no. 1, 55--56.
\bibitem{polignac} A. de Polignac, \emph{Recherches sur les nombres premiers}, Comptes Rendus Acad. Sci. Paris \textbf{45}, 406--410, 431--434, 575--580, 
882--886. 
\bibitem{rasa} O. Ramar\'{e} and Y. Saouter, \emph{Short effective intervals containing primes}, J. Number Theory \textbf{98} (2003), no.1, 10--33.
\bibitem{rosser1962} J. B. Rosser and L. Schoenfeld, \emph{Approximate formulas for some functions of prime numbers}, Illinois J. Math. \textbf{6}:1 (1962), 
64--94.
\bibitem{schoenfeld1976} L. Schoenfeld, \emph{Sharper bounds for the Chebyshev functions $\theta(x)$ and $\psi(x)$} II, Math. Comp. \textbf{30} (1976), no. 
134, 337--360.
\bibitem{sloane1} N. J. A. Sloane, Sequence A077761, The on-line encyclopedia of integer sequences. \url{https://oeis.org/A077761}. Accessed 17 August 2016.
\bibitem{sloane2} N. J. A. Sloane, Sequence A083343, The on-line encyclopedia of integer sequences. \url{https://oeis.org/A083343}. Accessed 17 August 2016.
\bibitem{trudgian2014} T. S. Trudgian, \emph{An improved upper bound for the argument of the Riemann zeta-function on the critical line} II, J. Number Theory 
\textbf{134} (2014), 280--292.
\bibitem{trud} T. S. Trudgian, \emph{Updating the error term in the prime number theorem}, Ramanujan J. \textbf{39} (2016), no. 2, 225--234.
\bibitem{vallee1896} C.-J. de la Vall\'{e}e Poussin, \emph{Recherches analytiques la th\'{e}orie des nombres premiers}, Ann. Soc. scient. Bruxelles \textbf{20}
(1896), 183--256.
\bibitem{vallee1899} C.-J. de la Vall\'{e}e Poussin, \emph{Sur la fonction $\zeta(s)$ de Riemann et le nombre des nombres premiers inf\'{e}rieurs \`{a} une
limite donn\'{e}e}, Mem. Couronn\'{e}s de l'Acad. Roy. Sci. Bruxelles \textbf{59} (1899), 1--74.
\bibitem{vinogradov1958} I. M. Vinogradov, \emph{A new estimate of the function $\zeta(1+it)$}, Izv. Akad. Nauk SSSR. Ser. Mat. \textbf{22}:2 (1958), 161--164.
\end{thebibliography}
\end{document}